\documentclass[12pt]{amsart}

\usepackage[numbers]{natbib}
\usepackage[utf8]{inputenc} 
\usepackage[T1]{fontenc} 
\usepackage[dvipsnames, table]{xcolor} 
\usepackage{color} 
\usepackage{tikz} 
\usetikzlibrary{backgrounds, positioning}
\usepackage{tikz-cd} 
\usetikzlibrary{decorations.pathmorphing}
\tikzcdset{
  cells={font=\everymath\expandafter{\the\everymath\displaystyle}},
} 
\usepackage{amssymb, amsmath, mathtools, amsfonts, amsthm} 
\usepackage{graphicx} 
\graphicspath{ {./images/} }
\usepackage{caption} 
\usepackage{latexsym} 
\usepackage[inline]{enumitem} 
\usepackage{hyperref} 
\hypersetup{colorlinks=true, urlcolor=caribbeangreen, citecolor=gray, linkcolor=caribbeangreen}
\usepackage{url} 
\usepackage{indentfirst} 
\usepackage{listings} 
\usepackage[mathscr]{eucal} 
\usepackage{manfnt} 
\usepackage{longtable} 
\usepackage{colonequals} 
\usepackage{nicefrac} 
\usepackage{caption} 
\usepackage{subcaption} 
\usepackage{float}
\usepackage[capitalize]{cleveref} 
\usepackage{stackengine} 
\usepackage{breakcites}
\usepackage{stmaryrd} 
\usepackage{bm}       

\newcounter{counter}[section] 
\counterwithin{counter}{section}
\newtheorem{theorem}[counter]{Theorem}
\newtheorem{lemma}[counter]{Lemma}
\newtheorem{corollary}[counter]{Corollary}

\newtheorem{proposition}[counter]{Proposition}

\theoremstyle{definition} 
\newtheorem{definition}[counter]{Definition}
\newtheorem{situation}[counter]{Situation}
\newtheorem{heuristic}[counter]{Heuristic}
\newtheorem{example}[counter]{Example} 
\newtheorem{notation}[counter]{Notation}
\newtheorem{remark}[counter]{Remark}

\definecolor{santicolor}{RGB}{0, 159, 183} 
\definecolor{headercolor}{rgb}{0.83, 0.83, 0.83} 
\definecolor{caribbeangreen}{rgb}{0.0, 0.8, 0.6}
\definecolor{acolor}{rgb}{0.01, 0.75, 0.24}
\definecolor{bcolor}{RGB}{0, 159, 183}
\definecolor{ccolor}{RGB}{254, 74, 73}
\definecolor{linkcolor}{rgb}{0.0, 0.0, 0.61}
\definecolor{codegreen}{rgb}{0,0.6,0}
\definecolor{codegray}{rgb}{0.5,0.5,0.5}
\definecolor{codepurple}{rgb}{0.58,0,0.82}
\definecolor{backcolour}{rgb}{0.95,0.95,0.92}
\lstdefinestyle{mystyle}{backgroundcolor=\color{backcolour}, commentstyle=\color
 {codegreen}, keywordstyle=\color{magenta}, numberstyle=\tiny\color{codegray}, stringstyle=\color{codepurple}, basicstyle=\ttfamily\footnotesize, breakatwhitespace=false, breaklines=true, captionpos=b, keepspaces=true, numbers=left, numbersep=4pt, showspaces=false, showstringspaces=false, showtabs=false, tabsize=2}
\newcommand{\R}{\mathbb{R}} 
\newcommand{\Z}{\mathbb{Z}} 
\newcommand{\Q}{\mathbb{Q}} 
\newcommand{\QQ}{\Q} 

\newcommand{\C}{\mathbb{C}} 
\newcommand{\Qbar}{{\bar{\Q}}} 
\newcommand{\kbar}{{\bar{k}}} 

\newcommand{\A}{\mathbb{A}} 

\newcommand{\kk}{{\mathbf k}} 
\renewcommand{\P}{\mathbb{P}} 
\newcommand{\Pone}{\P^1} 
\newcommand{\XX}{\mathscr{X}} 
\newcommand{\mcT}{\mathcal{T}} 
\newcommand{\mcS}{\mathcal{S}} 
\newcommand{\mcR}{\mathcal{R}} 
\newcommand{\mcU}{\mathcal{U}} 
\newcommand{\mcV}{\mathcal{V}} 

\newcommand{\belyi}{Belyi } 
\newcommand\HH{\mathrm{H}}

\newcommand{\coeff}{\mathbf{sfc}} 
\newcommand{\tri}{\triangle} 
\newcommand{\tribar}{\bar\triangle} 
\newcommand{\x}{\mathsf{x}}
\newcommand{\y}{\mathsf{y}}
\newcommand{\z}{\mathsf{z}}
\renewcommand{\t}{\mathsf{t}}
\newcommand{\s}{\mathsf{s}}

\newcommand{\abc}{(a,b,c)}

\renewcommand{\leq}{\leqslant}
\renewcommand{\geq}{\geqslant}
\newcommand{\md}{\operatorname{mod}} 

\newcommand{\ideal}[1]{\langle #1 \rangle} 
\newcommand{\paren}[1]{\left( #1 \right)} 
\newcommand{\cdef}[1]{{\color{linkcolor}\textsf{#1}}} 
\newcommand{\brk}[1]{\left\{ #1 \right\}} 

\renewcommand{\th}{\text{th}}

\newcommand{\gfe}{A\x^{a} + B\y^{b} + C\z^{c}}

\newcommand{\mif}{\text{ if }} 
\renewcommand{\emptyset}{\varnothing}

\DeclareMathOperator{\Spec}{Spec} 
\DeclareMathOperator{\Proj}{Proj} 
\DeclareMathOperator{\Gal}{Gal} 

\DeclareMathOperator{\Aut}{Aut} 
\DeclareMathOperator{\Autsch}{\mathbf{Aut}} 
\let\num\relax
\DeclareMathOperator{\num}{num}
\DeclareMathOperator{\den}{den}

\DeclareMathOperator{\PGL}{PGL}

\DeclareMathOperator{\Ht}{Ht}
\DeclareMathOperator{\vol}{vol} 

\usepackage[textsize=tiny]{todonotes}
\usepackage{xargs, xpatch}
\usepackage{marginnote}

\title[Counting primitive solutions]{Counting primitive integral solutions to spherical generalized Fermat equations}
\author{Santiago Arango-Pi{\~n}eros}
\address{Department of Mathematics, University of Massachusetts Amherst, Amherst, MA 01003, USA}
\email{santiago.arango.pineros@gmail.com}
\urladdr{\url{https://sarangop1728.github.io/}}

\begin{document}

\begin{abstract}
  A solution $(x,y,z) \in \mathbb{Z}^3-\brk{(0,0,0)}$ to a generalized Fermat equation
\begin{equation}
    \label{eq:abstract-gfe}
    \gfe = 0,
\end{equation}
is called \emph{primitive} if $\gcd(x,y,z) = 1$. By work of Beukers
\cite{Beukers98}, we know that in the \emph{spherical} regime (that is, when
the Euler characteristic
$\chi = \tfrac{1}{a} + \tfrac{1}{b} + \tfrac{1}{c} - 1$ is positive), if
\Cref{eq:abstract-gfe} has one primitive solution, then it has infinitely many.
In this work, we use the method of \emph{Fermat descent}, as employed by
Poonen--Schaefer-Stoll \cite{PoonenSchaeferStoll07}, to refine Beukers' result
to an asymptotic count of the number of primitive integral solutions of bounded
height. 
\end{abstract}

\maketitle
\hfill \emph{MSC class: 11D41}
\hrule
\vspace{-5mm}
\setcounter{tocdepth}{1}
\tableofcontents
\vspace{-11mm}


\section{Introduction}
\label{sec:introduction}

\subsection{Poonen's heuristic}
\label{sec:heuristic} We follow \cite{Poonen_slides2006}. Let $a,b,c$ be
positive integers, and consider the following subset of the rational points on
the projective line $\Pone(\Q) \cong \Q \cup \brk{\tfrac10}$.
\begin{align}
  \label{eq:Omega-abc}
  \Omega\abc \colonequals \brk{Q \in \Pone(\Q) :
  \begin{array}{l}
    \mathrm{(i)}~ \num(Q) \text{ is an $a^\th$ power}, \\
    \mathrm{(ii)}~ \num(Q-1) \text{ is a $b^\th$ power}, \\
    \mathrm{(iii)}~ \den(Q) \text{ is a $c^\th$ power}. 
  \end{array}
  }.
\end{align}
By the \cdef{numerator} and \cdef{denominator} of a point
$Q = (s:t) \in \Pone(\Q)$, we mean the first and second coordinate, respectively,
of a representative point $\pm (s,t) \in \Z^2$ with $\gcd(s,t) = 1$. This pair is
only well defined up to sign. We say that an integer $m$ is an \cdef{$n^\th$
  power} if the ideal $m\Z$ equals $e^n\Z$ for some $e \geq 0$. In particular,
$0,1,\infty \in \Omega\abc$.

To any subset $\Omega \subseteq \Pone(\Q)$ we associate the subset of points of
bounded height, and the corresponding counting function. Given $h$ positive, define
\begin{align}
  \Omega_{\leq h} \colonequals \brk{Q \in \Omega : \Ht(Q) \leq h}, \quad
  N(\Omega; h) \colonequals \# \Omega_{\leq h},
\end{align}
where $\Ht\colon \Pone(\Q) \to \Z_{\geq 0}$ is the usual multiplicative height,
given by
\begin{equation}
  \label{eq:P1-height}
  \Ht(Q) = \max\brk{|\num(Q)|,|\den(Q)|}.
\end{equation}

\begin{heuristic}
  \label{heuristic:poonen}
  We estimate the probability that a uniformly random rational
  number of height not exceeding $h \gg 0$ is in the set $\Omega\abc$. We do this under
  the heuristic assumption that the events (i), (ii), and (iii) defining
  $\Omega\abc$ in \Cref{eq:Omega-abc} are \emph{independent}.

  We have that
  \begin{align*}
    \dfrac{\#\brk{Q \in \Pone(\Q)_{\leq h} \colon \num(Q) \text{ is an $a^\th$
    power}}}{\# \Pone(\Q)_{\leq h}} &\doteq \dfrac{h\cdot h^{1/a}}{h^2} =
                                      h^{-1+1/a}, \\
    \dfrac{\#\brk{Q \in \Pone(\Q)_{\leq h} \colon \num(Q-1) \text{ is an $b^\th$
    power}}}{\# \Pone(\Q)_{\leq h}} &\doteq \dfrac{h\cdot h^{1/b}}{h^2} =
                                      h^{-1+1/b}, \\
    \dfrac{\#\brk{Q \in \Pone(\Q)_{\leq h} \colon \den(Q) \text{ is an $c^\th$
    power}}}{\# \Pone(\Q)_{\leq h}} &\doteq \dfrac{h\cdot h^{1/c}}{h^2} =
                                      h^{-1+1/c}, 
  \end{align*}
  where the notation $f(h) \doteq g(h)$ means that there exists an implicit
  constant $\kappa > 0$ such that $f(h) = \kappa\cdot g(h)$ as $h \to \infty$. The
  independence assumption implies that
  \begin{align*}
    \dfrac{\#\Omega\abc}{\# \Pone(\Q)_{\leq h}} \doteq
    \paren{h^{-1+1/a}}\paren{h^{-1+1/b}}\paren{h^{-1+1/c}} \doteq h^{-3+1/a+1/b+1/c}.
  \end{align*}
The heuristic above suggests that the \cdef{Euler characteristic}
  \begin{equation}
    \label{eq:characteristic}
    \chi\abc \colonequals \tfrac1a + \tfrac1b + \tfrac1c -1
  \end{equation}
\end{heuristic}
\noindent forces $\Omega\abc$ to be
\begin{align*}
  \begin{cases}
    \text{infinite}, & \mif \chi\abc > 0, \text{ and} \\
    \text{finite}, & \mif \chi\abc < 0.
  \end{cases}
\end{align*}
This prediction turns out to be correct. The \cdef{hyperbolic} case (when
$\chi < 0$) can be deduced from a theorem of Darmon and Granville \cite[Theorem
2]{Darmon&Granville95}. The \cdef{spherical} case (when $\chi > 0 $ and therefore
the multiset $\{a,b,c\}$ is one of $\{2,3,3\}, \{2,3,4\}, \{2,3,5\},$ or
$\{2,2,c\}$ for $c \geq 2$) the heuristic suggests that in the spherical case one
has $N(\Omega\abc;h) \asymp h^{\chi}$, as $h$ tends to infinity. This can be deduced from a
theorem of Beukers \cite[Theorem 1.2]{Beukers98}.

\subsection{Results}
\label{sec:intro-main-theorem} Our first result confirms the prediction of
\Cref{heuristic:poonen}.
\begin{theorem}
  \label{thm:main-belyi} Suppose that $a,b,c > 1$ and that
  $\chi \colonequals \chi\abc > 0$. Then, there
  exists an explicitly computable constant $\kappa\abc >0$ such that for every
  $\varepsilon > 0$, 
  \begin{align*}
    N(\Omega\abc; h) = \kappa\abc\cdot h^{\chi} + O(h^{\chi/2 + \varepsilon}), 
  \end{align*}
  as $h \to \infty$. The implicit constant depends on $\abc$ and
  $\varepsilon$.
\end{theorem}

Consider the generalized Fermat equation
\begin{equation}
  \label{eq:gfe}
  F\colon A\x^a + B\y^b + C\z^c = 0 \subset \A^3_\Z 
\end{equation}
for arbitrary integers $A,B,C$ satisfying $A\cdot B \cdot C \neq 0$. A solution
$(x,y,z) \in \Z^3 - \brk{(0,0,0)}$ is said to be \cdef{primitive} when
$\gcd(x,y,z) = 1$. Corresponding to each $F$, we have the \cdef{punctured cone}
$\mcU$ (obtained by deleting the closed subscheme $\brk{\x = \y = \z = 0}$ from
$F$) and the morphism
\begin{equation}
  \label{eq:j-map}
  j\colon \mcU \to \Pone_\Z, \quad (x,y,z) \mapsto (-Ax^a : Cz^c).
\end{equation}
Note that $\mcU(\Z)$ is identified with the set of primitive integral solutions
to $F$. Define the subset $\Omega(F) \subset \Pone(\Q)$ to be the image of the function
$j(\Z)\colon \mcU(\Z) \to \Pone(\Z) = \Pone(\Q)$.

The set $\Omega(a,b,c)$ and the primitive integral solutions to the equation
are closely related when $A,B,C \in \Z^\times = \brk{\pm 1}$. Indeed, given $Q\in
\Omega(a,b,c)$, then $|\num(Q)| = |x|^a$, $|\num(Q-1)| = |y|^b$ and $|\den(Q)| =
|z|^c$. From the identity
\[-\num(Q) + \num(Q-1) + \den(Q) = 0,\] we deduce that $(x,y,z)$ is a primitive
integral solution to \Cref{eq:gfe} for some choice of
$(A,B,C)\in \brk{\pm 1}^3$. Conversely, given a primitive integral
solution $(x,y,z)$ to the equations
\[\x^a + \y^b + \z^c = 0, \quad \x^a + \y^b - \z^c = 0, \quad \x^a - \y^b +
  \z^c = 0,\quad \x^a - \y^b - \z^c = 0,\] we see that $Q = \pm x^a/z^c$ is in
$\Omega(a,b,c)$.

By carefully identifying how the sets $\Omega(F)$ fit inside of $\Omega\abc$
(or rather, certain supersets $\Omega_\mcS\abc \supset \Omega\abc$) we are able
to obtain the following stronger result.
 
\begin{theorem}
  \label{thm:main-fermat} Consider \Cref{eq:gfe} with $A,B,C \in \Z$ nonzero
  and $a,b,c > 1$. Suppose that $\chi\colonequals \chi\abc >0$, and that there
  exists at least one primitive integral solution to $F$. Then, there exists an
  explicit constant $\kappa(F) >0$ such that for every $\varepsilon > 0$,
  \begin{align*}
    N(\Omega(F); h) = \kappa(F)\cdot h^{\chi} + O(h^{\chi/2 + \varepsilon}),
  \end{align*}
 as $h \to \infty$. The implied constant depends on $\varepsilon$.
\end{theorem}

\subsection{$\mcS$-integral points on the \belyi stack}
\label{sec:proof-sketch} Our approach is geometric. We use the method of
\emph{Fermat descent}, developed by \cite{Darmon&Granville95}, \cite{Darmon97},
and \cite{PoonenSchaeferStoll07}, and expanded on in \cite{SAP-fermat-descent}
from the point of view of stacks. It turns out $\Omega\abc$ is precisely the set of
$\Z$-points on the \cdef{\belyi stack of signature $\abc$}, denoted by
$\Pone\abc$ (see \cite[Section 3]{SAP-fermat-descent}). This is the stacky
version of Darmon's $M$-curve $\mathbf{P}^1_{a,b,c}$ \cite[p.~4]{Darmon97}.

\begin{notation}[Set of points on a stack]
  If $\XX$ is a stack and $R$ is a ring, we denote by $\XX(R)$ the
  \emph{groupoid} of $R$-points, and by $\XX\ideal{R}$ the \emph{set} of
  $R$-points, (see \cite[Section 2.1]{SAP-fermat-descent}).
\end{notation}

For the purposes of this work, we need only to understand the set
$\Pone\abc\ideal{R}$ in the case that $R = \Z[\mcS^{-1}]$ for some finite set
of rational primes $\mcS$. In \cite[Lemma 3.3]{SAP-fermat-descent}, we show
that the set $\Pone\abc\ideal{R}$ is in bijective correspondence with the
subset $\Omega_\mcS\abc$ of the rational points on the projective line
$Q\in \Pone(\Q)$ which satisfy the property that the ideals
\begin{equation}
  \label{eq:Omega-S-abc}
  \num(Q)R, \quad \num(Q-1)R, \quad \den(Q)R,
\end{equation}
are $a^\th$, $b^\th$, and $c^\th$ powers respectively. Since $R$ is a principal
ideal domain, $Q$ belongs to $\Omega_\mcS\abc \subset \Pone(\Q)$ if and only if
\[
\num(Q) = -Ax^a, \quad \num(Q-1) = By^b, \quad \den(Q) = Cz^c,
\]
for some $A, B, C \in R^\times$, and $x,y,z\in \Z$ with $\gcd(x,y,z) = 1$. This choice
of coefficients $(A,B,C)\in (R^\times)^{3}$ is only well defined up to coordinate-wise
multiplication by a unit in $R$. In particular, we can arrange for $A,B,C$ to
be in $\Z \cap R^\times = \brk{n \in \Z: p\mid n \text{ implies } p \in \mcS}$. These
considerations lead to the following definition.
\begin{definition}
  \label{def:fermat-coefficients}
  Let $\mcS$ be a finite set of primes. Define the \cdef{$\mcS$-simplified
    Fermat coefficient triple} of a point $Q \in \Omega_\mcS\abc$ to be the unique
  triple $(A,B,C) \in \Z^3$ satisfying the following properties:
  \begin{enumerate}[leftmargin=*, label=(\roman*)]
  \item \label{it:R-units} The integers $A,B,C$ are $\mcS$-units (i.e., $A,B,C
    \in \Z[\mcS^{-1}]^{\times}\cap \Z$.)
  \item \label{it:power-free} $A$ is $a^\th$ power-free, $B$ is $b^\th$ power-free, and $C$ is
    $c^\th$ power-free.
  \item \label{it:normalization} $\gcd(A,B,C) = 1$ and $A > 0$.
  \item \label{it:div} $A \mid \num(Q)$, $B \mid \num(Q-1)$, and $C \mid \den(Q)$.
  \end{enumerate}
  We denote this assignment by $\coeff(Q) = (A,B,C)$. We say that a generalized
  Fermat equation $F\colon \gfe = 0$ is \cdef{$\mcS$-simple} or
  \cdef{$\mcS$-simplified} if the coefficients $(A,B,C)$ satisfy the properties
  \ref{it:R-units}, \ref{it:power-free}, and \ref{it:normalization} above. We say that
  $F$ is \cdef{simple} or \cdef{simplified}, if it is $\mcS$-simple for
  $\mcS \colonequals \brk{p \text{ prime}: p \mid A\cdot B \cdot C}$.
\end{definition}
\begin{example}
  The $\emptyset$-simple Fermat equations have $\pm 1$ coefficients.
\end{example}

\subsection{The Pythagorean case}
\label{sec:x2+y2-z2} To introduce the main ideas in our proofs, we consider the elementary case
of signature $\abc = (2,2,2)$, where the mention of stacks is unnecessary and
could be considered excessive. We remark that Lehmer \cite[p.~38]{Lehmer1900}
and Lambek--Moser \cite{Lambek&Moser55} already counted the number of
Pythagorean triangles with bounded hypotenuse, and the analytic number theory
techniques used in their work and in ours remain essentially the same. In our
notation, their theorem would read as follows.

\begin{theorem}[Lehmer, Lambek--Moser]
  \label{thm:main-fermat-222} Consider the Pythagorean equation $F_3\colon
  \x^2+\y^2-\z^2 = 0$. Then, 
  \begin{align*}
    N(\Omega(F_3); h) \sim \tfrac{1}{\pi}\cdot h^{1/2},
  \end{align*}
 as $h \to \infty$. 
\end{theorem}

We observe that \Cref{thm:main-fermat-222} implies the following consequence of
\Cref{thm:main-belyi}.
\begin{theorem}
  \label{thm:main-belyi-222} The asymptotic
  \begin{align*}
    N(\Omega(2,2,2); h) \sim \tfrac{3}{\pi}\cdot h^{1/2} 
  \end{align*}
  holds, as $h \to \infty$.
\end{theorem}
\begin{proof}
  Consider the group $G \colonequals \brk{\pm 1}^3/\langle\pm 1\rangle$, and list its elements
\begin{align*}
  \label{eq:H}
  e_0 = [1,1,1], \quad  e_1 = [1,-1,-1], \quad  e_2 = [1,-1,1], \quad  e_3 = [1,1,-1].
\end{align*}
Consider the Fermat conics $F_0, F_1, F_2, F_3$ with $\x^2,\y^2,\z^2$
coefficients given by the element in $G$ with matching index. For each element in
$G$, we attach a corresponding map $j\colon \mcU \to \Pone$ as in \Cref{eq:j-map}.

\begin{table}[H]
\caption{$G$-twists of Pythagorean equation.}
\label{table:twists-pythagorean-equation}
\setlength{\arrayrulewidth}{0.2mm} \setlength{\tabcolsep}{5pt}
\renewcommand{\arraystretch}{2}
  \begin{tabular}{|c|c|l|}
    \hline
    \rowcolor{headercolor}
    $G$   & $F$            & \multicolumn{1}{c|}{$j$}  \\ \hline
    $e_0$ & $\x^2+\y^2+\z^2 = 0$ & $(x,y,z) \mapsto (-x^2:z^2)$ \\ \hline
    $e_1$ & $\x^2-\y^2-\z^2 = 0$ & $(x,y,z) \mapsto (x^2:z^2)$  \\ \hline
    $e_2$ & $\x^2-\y^2+\z^2 = 0$ & $(x,y,z) \mapsto (-x^2:z^2)$ \\ \hline
    $e_3$ & $\x^2+\y^2-\z^2 = 0$ & $(x,y,z) \mapsto (x^2:z^2)$  \\ \hline
\end{tabular}
\end{table}
The set $\Omega(2,2,2)$ is the pushout
  \begin{equation*}
    \dfrac{\Omega(F_1)\sqcup \Omega(F_2) \sqcup \Omega(F_3)}{\brk{0,1,\infty}}.
  \end{equation*}
  In other words,
  $\Omega(2,2,2) = \Omega(F_1)\cup \Omega(F_2) \cup \Omega(F_3)$ and the pairwise intersections
  $\Omega(F_i)\cap \Omega(F_j)$ for $i,j \in \brk{1,2,3}$ are contained in
  $\brk{0,1,\infty}$. This can be checked by partitioning the set
  $\Omega(2,2,2)$ according to the signs of $\num(Q), \num(Q-1),$ and
  $\den(Q)$, and staring at \Cref{table:twists-pythagorean-equation}. From this
  description, we deduce from \Cref{thm:main-fermat-222} that
  \[N(\Omega(2,2,2);h) = N(\Omega(F_1);h) + N(\Omega(F_2);h) + N(\Omega(F_3);h)
    \sim \tfrac{3}{\pi}\cdot h^{1/2}.\]
\end{proof}

Now, we will prove \Cref{thm:main-fermat-222} using the method of \emph{Fermat
  descent}.

\begin{proof}[Proof of \Cref{thm:main-fermat-222}] The proof proceeds in three
  steps: covering, twisting, and sieving. \\
  
  \noindent \emph{Step 1: (Covering)} A suitable covering is readily available. Indeed, if
  $Z_0$ denotes the plane conic defined by $F_0$, the $j$-map
  $j_0\colon \mcU_0 \to \Pone$ induces the morphism
\[
  \phi_0 \colon Z_0 \to \Pone_\Q, \quad (x:y:z) \mapsto (-x^2:z^2).
\]
One verifies that $\phi$ is a Galois \belyi map defined over $\Q$ with Galois
group $G$, diagonally embedded in $\PGL_3(\Q)$. Since $\mcU_0(\Z)$ is empty, so
is $\Omega(F_0)$.

Any other cover $\phi_i\colon Z_i \to \Pone$ (induced from
$j_i\colon \mcU_i \to \Pone$) would suffice, but we choose the pointless conic
for dramatic emphasis.

  \medskip
  \noindent\emph{Step 2: (Twisting)} Consider the Galois cohomology group $\HH^1(\Q, G)$.
  Since the absolute Galois group $\Gal_\Q \colonequals \Gal(\Qbar | \Q)$ acts
  trivially on the abelian group $G$, $\HH^1(\Q,
  G)$ 
  is the group of continuous group homomorphisms $\Gal_\Q \to G$. Every such
  map factors through a unique injective morphism
  $\Gal(L|\Q) \hookrightarrow G$, where $L \supset \Q$ is a finite Galois
  extension.

  The only bad prime for the covering (in the sense of \cite[Lemma
  3.23]{SAP-fermat-descent}) $\phi$ is $p=2$. In the notation of
  \Cref{sec:proof-sketch}, $\mcS = \brk{2}$, and $R = \Z[1/2]$. By descent
  theory, we are only interested in the subset
  $\HH^1_\mcS(\Q,G) \subset \HH^1(\Q, G)$ corresponding to those injections
  $\Gal(L|\Q) \hookrightarrow G$ for which $L$ is unramified outside $\{2\}$.
  The possible fields are
  $$L \in \brk{\Q, \Q(\sqrt{-1}), \Q(\sqrt{2}), \Q(\sqrt{-2}), \Q(\zeta_8)}.$$

  Descent theory tells us that the set
  $\Omega_\mcS(2,2,2) \colonequals \Pone(2,2,2)\ideal{R} \cong
  [\Pone_R/\Autsch(\Phi)]\ideal{R}$ is partitioned by the disjoint union of the
  sets $\phi_\rho(Z_\rho(\Q))$, as $\rho$ ranges over $\HH^1_\mcS(\Q, G)$.
  \begin{equation}
    \label{eq:descent-partition-222}
    \Omega_\mcS(2,2,2) = \bigsqcup_{\rho\in \HH_\mcS^1(\Q,G)} \phi_\rho(Z_\rho(\Q)).
  \end{equation}
  It is well known that for a finite morphism
  $\phi\colon \Pone_\Q \to \Pone_\Q$, one has that
  $N(\phi(\Pone(\Q)); h) \asymp h^{2/\deg \phi}$. Moreover, in the special case
  that $\phi$ is geometrically Galois,
  $N(\phi(\Pone(\Q)); h) \sim \kappa(\phi)\cdot h^{2/\deg \phi}$ for some
  explicitly computable constant $\kappa(\phi)>0$. We give a detailed proof of
  these results in \Cref{sec:count-rat-points} for completeness. Combining this
  with the partition \Cref{eq:descent-partition-222} implies that
  \begin{equation*}
    N(\Omega_\mcS(2,2,2);h) = \sum_{\rho} N(\phi_\rho(Z_\rho(\Q)); h) \sim
    \kappa((2,2,2), \mcS)\cdot h^{1/2},
  \end{equation*}
  where the sum is restricted to those $\rho\colon \Gal(L|\Q) \hookrightarrow
  G$ in $\HH_\mcS(\Q, G)$ for which the twist $Z_\rho$ is isomorphic to
  $\Pone_\Q$. In particular, the constant $\kappa((2,2,2),\mcS)$ will be the
  sum of the constants $\kappa(\phi_\rho)$.

  \medskip
  \noindent\emph{Step 3: (Sieving)} The count above already contains
  the count of the proper subset $\Omega(F_3) \subset \Omega_\mcS(2,2,2)$ that
  we seek. Indeed, starting from the
  partition~(\ref{eq:descent-partition-222}), we note that, since the twists
  $\phi_\rho$ are (Galois) \belyi maps of signature $(2,2,2)$, we can assign to
  each $\rho \in \HH^1_\mcS(\Q, G)$ a unique $2$-simplified coefficient
  $(A_\rho, B_\rho, C_\rho)$ such that $\phi_\rho(Z_\rho(\Q))$ is contained in
  the set $\Omega(F_\rho)$, associated to the generalized Fermat equation
  \[F_\rho\colon A_\rho\x^2+B_\rho\y^2+C_\rho\z^2 = 0.\] In particular, we
  deduce that some twist of $\phi_0\colon Z_0 \to \Pone$ is isomorphic to
  $\phi_3\colon Z_3 \to \Pone$, and that
  $\Omega(F_3) = \Omega(\phi_3(Z_3(\Q)))$. In
  \Cref{example:pythagorean-constant}, we calculate that $\kappa(F_3) = 1/\pi$,
  and we conclude that
  \begin{equation*}
    N(\Omega(F_1);h) = N(\Omega(F_2);h) = N(\Omega(F_3);h) \sim \tfrac1\pi\cdot h^{1/2}.
  \end{equation*}
\end{proof}

\subsection{Previous work on spherical Fermat equations}
\label{sec:previous-work}
This work is closely related to, and inspired by, the foundational
contributions of Beukers~\cite{Beukers98}. Indeed, the arguments in
\Cref{sec:proofs} can be slightly modified to
reprove~\cite[Theorem~1.2]{Beukers98}. On a related note, the excellent
Master's thesis of Esmonde~\cite{Esmonde99} addresses the problem of solving
the equation $\x^a + \y^b - \z^c = 0$ in polynomial rings $k[t]$, for certain
examples of fields~$k$. Building on work of Beukers, Edwards~\cite{Edwards04}
completed the parametrizations of the spherical equations $\x^2+\y^3-\z^3=0$,
$\x^2+\y^3-\z^4 = 0$, and $\x^2+\y^3-\z^5 = 0$. We expect that the method of
Fermat descent employed here can be extended to compute parametrizations for
general spherical Fermat equations; this is work in progress by the author.

\subsection*{Acknowledgments}
\label{sec:acknowledgements}
This work is part of the author’s PhD thesis. We thank David Zureick-Brown,
John Voight, and Andrew Kobin for many enlightening conversations on this topic
and for their valuable feedback. We are also grateful to Bjorn Poonen for
agreeing to serve on the thesis committee and for his detailed and insightful
comments on an earlier draft.

\section{Belyi maps and triangle groups}
\label{sec:preliminaries}

\subsection{(Spherical) triangle groups}
\label{sec:spherical-groups}
We follow \cite[Section 2]{Clark&Voight19}. For more on this topic see
\cite[Chapter II]{Magnus74}.

Let $a,b,c > 1$ be positive integers. We say that the triple $\abc$ is
\cdef{spherical}, \cdef{Euclidean}, or \cdef{hyperbolic} according as the
quantity
\begin{align*}
  \chi\abc \colonequals \tfrac1a + \tfrac1b + \tfrac1c - 1
\end{align*}
is positive, zero, or negative.

\begin{definition}
  \label{def:triangle-group} Given integers $a,b,c > 1$, the
  \cdef{extended triangle group} $\tri\abc$ is defined as the group generated
  by elements $\delta_0, \delta_1, \delta_\infty, -1$, with $-1$ central in
  $\tri\abc$, subject to the relations $(-1)^2 = 1$ and
    \begin{equation}
        \label{eq:triangle-relations}
        \delta_0^a = \delta_1^b = \delta_\infty^c = \delta_0\delta_1\delta_\infty = -1.
      \end{equation}
   Define the \cdef{triangle group} $\tribar\abc$ as the quotient of $\tri\abc$
   by $\brk{\pm 1}$.
\end{definition}

The spherical triangle groups are all finite groups. Moreover, they are all
finite subgroups of $\PGL_2(\Qbar)$. These were classified by Klein more than a
century ago. By \cite[Remark 2.2]{Clark&Voight19}, reordering the signature
$\abc$ to be nondecreasing $a \leq b \leq c$ does not affect the isomorphism
class of $\tribar\abc$.
\begin{itemize}[leftmargin=*]
\item For the \cdef{dihedral signatures} $\abc = (2,2,c)$ with $c\geq 2$, the
  triangle groups $\tribar(2,2,c)$ are isomorphic to the dihedral group $D_{c}$
  with $2c$ elements. In particular, $\tribar(2,2,3)$ is isomorphic to the
  symmetric group in three letters $S_3$. The group $\tribar(2,2,2)$ is
  isomorphic to the Klein four group $C_2 \times C_2$.
\item For the \cdef{tetrahedral signature} $\abc = (2,3,3)$, the triangle group
  $\tribar(2,3,3)$ is isomorphic to $A_4$; the group of rotational symmetries of the
  tetrahedron.
\item For the \cdef{octahedral signature} $\abc = (2,3,4)$, the triangle group
  $\tribar(2,3,4)$ is isomorphic to $S_4$; the group of rotational symmetries of the
  octahedron.
\item For the \cdef{icosahedral signature} $\abc = (2,3,5)$, the triangle group
  $\tribar(2,3,5)$ is isomorphic to $A_5$; the group of rotational symmetries of the
  icosahedron.
\end{itemize}

\begin{table}[ht]
\centering
\caption{Spherical triangle groups.}
\label{table:spherical-groups}
\setlength{\arrayrulewidth}{0.2mm} 
\setlength{\tabcolsep}{5pt}
\renewcommand{\arraystretch}{2}
\begin{tabular}{|c|c|c|}
  \hline
  \rowcolor{headercolor}
  $\abc$ & $\tribar\abc$ & $\chi\abc$ \\ \hline
  $(2,2,c)$ & $D_c$ & $1/c$ \\ \hline
  $(2,3,3)$ & $A_4$ & $1/6$ \\ \hline
  $(2,3,4)$ & $S_4$ & $1/12$ \\ \hline
  $(2,3,5)$ & $A_5$ & $1/30$ \\ \hline
\end{tabular}
\end{table}

\subsection{(Spherical) Belyi maps}
\label{sec:spherical-belyi-maps} By \cdef{curve} we mean a separated scheme of
finite type over a field of dimension one. We say that a curve is \cdef{nice}
if it is smooth, projective, and geometrically irreducible.

\begin{definition}
	\label{def:belyi-map}
	Let $Z_k$ be a nice curve defined over a perfect field $k$. A
    \cdef{$k$-\belyi map} is a finite $k$-morphism
    $\phi\colon Z_k \to \Pone_k$ that is unramified outside
    $\brk{0, 1, \infty} \subset \Pone(k)$.
  \end{definition}

  \begin{remark}
    These remarkable covers of the projective line are named after the
    Ukrainian mathematician G. V. \belyi, who famously proved that a complex
    algebraic curve can be defined over a number field if and only if it admits
    a $\C$-\belyi map \cite{Belyi79, Belyi02}. For this reason, it is customary
    to require that $k \subset \C$ to use the term $\belyi$ map. We ignore this
    convention, and allow $k$ to be perfect of positive characteristic.
  \end{remark}

\begin{definition}
    \label{def:galois-belyi-map}
    Let $\phi\colon Z_k \to \Pone_k$ be a $k$-\belyi map with automorphism
    $k$-group scheme $\Aut(\phi)$. We say that $\phi$ is \cdef{geometrically
      Galois} if the extension of function fields
    $\kk(Z_\kbar) \supset \kk(\Pone_\kbar)$ is Galois, with \cdef{Galois group}
    denoted by $\Gal(\phi)$. Equivalently, $\phi$ is geometrically Galois if
    $\Autsch(\phi)(\kbar) = \Aut(\phi_\kbar)$ acts transitively on the fibers. This is the case if
    and only if $\Aut(\phi_\kbar) \cong \Gal(\phi)$.
  \end{definition}
  \begin{remark}
    If $\phi\colon Z_k \to \Pone_k$ is a geometrically Galois $k$-\belyi map,
    for any $Q \in \Pone(k)-\brk{0,1,\infty}$, the fiber
    $\phi^{-1}(Q) \colonequals Z \times_k Q$ is a $\Gal(\phi)$-torsor over
    $\Spec k$.
  \end{remark}

  \begin{definition}
    The \cdef{signature} of a geometrically Galois $k$-\belyi map
    $\phi\colon Z_k \to \Pone_k$ is the triple $(e_0, e_1, e_\infty)$ where
    $e_P$ is the ramification index $e_\phi(z)$ of any critical point
    $z \in Z_k$ with critical value $P \in \brk{0,1,\infty}$. The \cdef{Euler
      characteristic} of $\phi$ is the quantity
    \begin{equation}
    \label{eq:euler-char}
    \chi(\phi) \colonequals \tfrac{1}{e_0} + \tfrac{1}{e_1} + \tfrac{1}{e_\infty} - 1.    
\end{equation} 
\end{definition}

As a consequence of the Riemann Existence Theorem, there exist Galois \belyi
maps of any spherical signature. See \cite[Proposition
3.1]{Darmon&Granville95} and \cite[Lemma 2.5]{Poonen05} for a proof of the
following proposition.

\begin{proposition}
    \label{prop:belyi-Galois-cover}
    For any positive integers $a,b,c > 1$, there exists a number field $K$ and
    a geometrically Galois $K$-\belyi map $\phi\colon Z_K \to \Pone_K$ of
    signature $(e_0,e_1,e_\infty) = \abc$. Let $g$ be the genus of $Z_K$, and $G$
    be the Galois group of $\phi$. Then
    $2-2g = (\deg\phi)\cdot \chi(\phi)$. In particular,
    \begin{enumerate}[label=(\roman*)]
    \item If $\chi(\phi) > 0$, then $g = 0$ and
      $\deg \phi = \# G = 2/\chi(\phi)$.
    \item If $\chi(\phi) = 0$, then $g = 1$.
    \item If $\chi(\phi) < 0$, then $g > 1$.
    \end{enumerate}
  \end{proposition}

  A crucial fact that we will need later is that for every spherical signature
  $\abc$, there exists a geometrically Galois \belyi map defined over $\Q$ with
  signature $\abc$. The reader may find several examples in the \belyi maps
  \href{https://beta.lmfdb.org/Belyi/}{LMFDB beta database} \cite{lmfdb:beta}.
  The maps presented in \Cref{table:spherical-groups} are adapted from the
  parametrizations found in \cite[Chapter 14]{Cohen07}. The original sources
  are \cite{Beukers98} and \cite{Edwards04}.
  \begin{table}[ht]
  \centering
  \caption{Examples of geometrically Galois $\Q$-\belyi maps for the spherical
    signatures.}
  \label{table:belyi-maps}
  \setlength{\arrayrulewidth}{0.2mm} \setlength{\tabcolsep}{5pt}
  \renewcommand{\arraystretch}{2}
  \begin{tabular}{|c|c|c|}
    \hline
    \rowcolor{headercolor}
    $\abc$ & $\tribar\abc$ & Example  \\ \hline
    $(2,2,c)$ & $D_c$ &  $\frac{(\s^c+\t^c)^2}{4(\s\t)^c}$ \\ \hline
    $(2,3,3)$ & $A_4$ & $\frac{(\s^2 - 2\s\t - 2\t^2)^2(\s^4 + 2\s^3\t + 6\s^2\t^2 - 4\s\t^3 + 4\t^4)^2}{2^6\t^3(\s-\t)^3(\s^2+\s\t+\t^2)^3}$ \\ \hline
    $(2,3,4)$ & $S_4$ & $\frac{-(4\s\t)^2(\s^2-3\t^2)^2(\s^4+6\s^2\t^2+81\t^4)^2(3\s^4+2\s^2\t^2+3\t^4)^2}{(s^2+3\t^2)^4(s^4-18\s^2\t^2+9\t^4)^4}$ \\ \hline
    $(2,3,5)$ & $A_5$ & $\frac{-(3^4\s^{10}+2^8\t^{10})^2(3^8\s^{20}-2^73^{10}\s^{15}\t^{5}-2^{18}3^{10}\s^{10}\t^{10}+2^{12}3^{10}\s^5\t^{15}+2^{16}\t^{20})^2}{(12\s\t)^5(81\s^{10} - 1584\s^5\t^5 - 256\t^{10})^5}$ \\ \hline
\end{tabular}
\end{table}

\section{Counting rational points in the image of a rational function}
\label{sec:count-rat-points}
The results presented in this section are undoubtedly well known
(\cite[p.~133]{Serre97-LecturesMordellWeil}, \cite[Theorem
B.6.1]{Hindry&Silverman00}); however, authors often ignore the leading
constants we seek. For completeness, we provide full proofs, making
the leading constants explicit.
\begin{situation}
  \label{situation:rat-pts-bounded-height-image-rational-map}
  Throughout the remainder of this section, we shall work with the following
  notations.
  \begin{itemize}[leftmargin=*]
  \item Let $\phi\colon \Pone_\Q \to \Pone_\Q$ be a nonconstant $\Q$-morphism with $d \colonequals \deg(\phi)$.
  \item Let $\phi_0, \phi_\infty \in \Z[\s,\t]$ be a choice of
    relatively prime homogeneous polynomials of degree $d$ such that $\phi$
    is given by 
    \begin{equation*}
      \phi(s:t) = (\phi_0(s,t):\phi_\infty(s,t)).
    \end{equation*}
  \item Let $\mcV \colonequals \A^2 - \mathbf{0}$ be the punctured cone over
    $\Pone_\Z$. We identify $\mcV(\Z)$ with the set
    $\brk{(s,t) \in \Z^2: \gcd(s,t) = 1}$. The map
    $\mcV(\Z) \to \Pone(\Q)$ given by $(s,t) \mapsto (s:t)$ is two-to-one.
  \item Denote by $\tilde \phi\colon \A^2 \to \A^2$ the lift $\tilde\phi (s,t)
    \colonequals (\phi_0(s,t),\phi_\infty(s,t))$ of $\phi$.    
  \item On $\Pone(\Q) = \Pone(\Z)$, $\Ht\colon \Pone(\Q) \to \Z_{\geq 0}$ is
    the usual multiplicative height, given by $\Ht(Q) = \max\brk{|\num(Q)|,|\den(Q)|}$.
  \item $\Omega(\phi) \subset \Pone(\Q)$ is the image of $\phi(\Q)\colon
    \Pone(\Q) \to \Pone(\Q)$.
  \item For any $\Omega \subset \Pone(\Q)$ and for every $h > 0$,
    $\Omega_{\leq h}$ is the finite subset of $\Omega$ consisting of those points
    $Q$ with $\Ht(Q) \leq h$. The \cdef{counting function of
      $\Omega \subset \Pone(\Q)$} is denoted $N(\Omega;h) \colonequals \# \Omega_{\leq h}$.
  \item We denote by $\Aut(\phi)$ the group of $\Q$-automorphisms of the map
    $\phi$.
  \end{itemize}
\end{situation}

The main result of this section is the following.

\begin{proposition}
  \label{prop:N(Omega;h)-asymptotics}
  We have $N(\Omega(\phi);h) \asymp h^{2/d}$ as $h\to \infty$. More precisely, there exists an
  explicitly computable constant $\delta(\phi) > 0$ such that
  \begin{equation*}
    \tfrac1d \cdot \delta(\phi)\cdot h^{2/d} \leq N(\Omega(\phi));h) \leq  \delta(\phi)\cdot h^{2/d}, \quad \text{ as } h \to \infty.
  \end{equation*}
  The constant $\delta(\phi)$ is described in \Cref{constant-term}.
\end{proposition}

In the special case where $\phi$ is geometrically Galois, we can keep track of
the exact number of $\Q$-rational points on each fiber
$\phi^{-1}(Q) \colonequals \Pone\times_\Q Q$, for all but finitely many
$Q \in \Omega(\phi)$. This allows us to promote the asymptotic bounds of
\Cref{prop:N(Omega;h)-asymptotics} to an asymptotic count.
\begin{corollary}
  \label{cor:counting}
  Suppose that $\phi$ is geometrically Galois. Then, there exists an explicitly
  computable constant $\kappa(\phi) \in \R_{>0}$ such that for every
  $\varepsilon > 0$,
  \[N(\Omega(\phi);h) = \kappa(\phi)\cdot h^{2/d} + O\paren{h^{1/d +
        \varepsilon}}\] as $h \to \infty$. Moreover, the leading constant is
  given by
  \begin{equation*}
    \kappa(\phi) = \delta(\phi)/\#\Aut(\phi),
  \end{equation*}
  and the implied constant depends on $\phi$ and $\varepsilon$.
\end{corollary}

\subsection{The primitivity defect set}
\label{sec:prim-defect-set}
Given $(s,t) \in \mcV(\Z)$, it does not follow that
$\tilde\phi(s,t) = (\phi_0(s,t), \phi_\infty(s,t))\in \mcV(\Z)$. For example,
consider the map
\[\tilde\phi(s,t) = ((s^2-t^2)^2, (s^2+t^2)^2)\] arising in the parametrization of
Pythagorean triples. When $s$ and $t$ have the same parity,
$\gcd \tilde\phi(s,t) = 4$. In general, $\tilde\phi\colon \mcV(\Z) \to \Z^2$ and
we have the following commutative diagram of sets.
\begin{equation*}
  \begin{tikzcd}
    \Pone(\Q) \arrow[dd, "\phi"'] & \mcV(\Z) \arrow[l] \arrow[dr, "\tilde\phi"] & \\
    & & \Z^2 \arrow[dl, "\cdot(1/\gcd)"]\\
    \Pone(\Q) & \mcV(\Z) \arrow[l]
  \end{tikzcd}
\end{equation*}

Define the \cdef{primitivity defect set of $\phi$} by
\begin{equation}
  \label{eq:primitivity-defect-set}
  \mathcal{D}(\phi) \colonequals \brk{\gcd \tilde\phi(s,t) : (s,t) \in \mcV(\Z)}.
\end{equation}
The set $\mathcal{D}(\phi)$ is finite. Indeed, let $R(\phi) \in \Z$ denote the resultant
of the homogeneous polynomials $\phi_0$ and $\phi_\infty$. Then, every
primitivity defect divides $R(\phi)$.
\begin{lemma}
  \label{lemma:primitivity-defect-set}
  If $e \in \mathcal{D}(\phi)$, then $e \mid R(\phi)$.
\end{lemma}
\begin{proof}
  Let $e \in \mathcal{D}(\phi)$. By definition, there exists
  $(s,t)\in \mcV(\Z)$ such that $\gcd \tilde\phi(s,t) = e$. By standard properties
  of the resultant, we can find polynomials $g_0,g_\infty \in \Z[\s,\t]$ such that
  \[ R(\phi) = g_0(\s,\t)\cdot \phi_0(\s,\t) +
    g_\infty(\s,\t)\cdot\phi_\infty(\s,\t).\]
  By evaluating the expression above at $(\s,\t) = (s,t)$, we see that $R(\phi)$
  is a multiple of $e$.
\end{proof}
For each $e \in \mathcal{D}(\phi)$, consider the set
  \begin{equation*}
    \mcV(\Z)_e \colonequals \brk{(s,t) \in \mcV(\Z) : \gcd\tilde\phi(s,t) = e}.
  \end{equation*}
  We have a partition
  \begin{equation}
    \label{eq:VZ-partition}
      \mcV(\Z) = \bigsqcup_{e\in \mathcal{D}(\phi)}  \mcV(\Z)_e.
    \end{equation}
      For each $e\in \mathcal{D}(\phi)$, consider the subsets
  \begin{equation*}
    \Z\cdot \mcV(\Z)_e \colonequals \brk{(ns,nt) : n\in \Z, (s,t)\in \mcV(\Z)_e}
    \subset \Z^2.
  \end{equation*}
  From the partition \Cref{fig:VZ-partition} of primitive points, we obtain the
  partition
  \begin{equation}
    \label{eq:ZZ-partition}
      \Z^2 = \bigsqcup_{e\in \mathcal{D}(\phi)}\Z\cdot \mcV(\Z)_e.
  \end{equation}

  \begin{figure}[ht]
    \centering
    \begin{subfigure}{0.49\textwidth} 
        \centering
        \includegraphics[width=\textwidth]{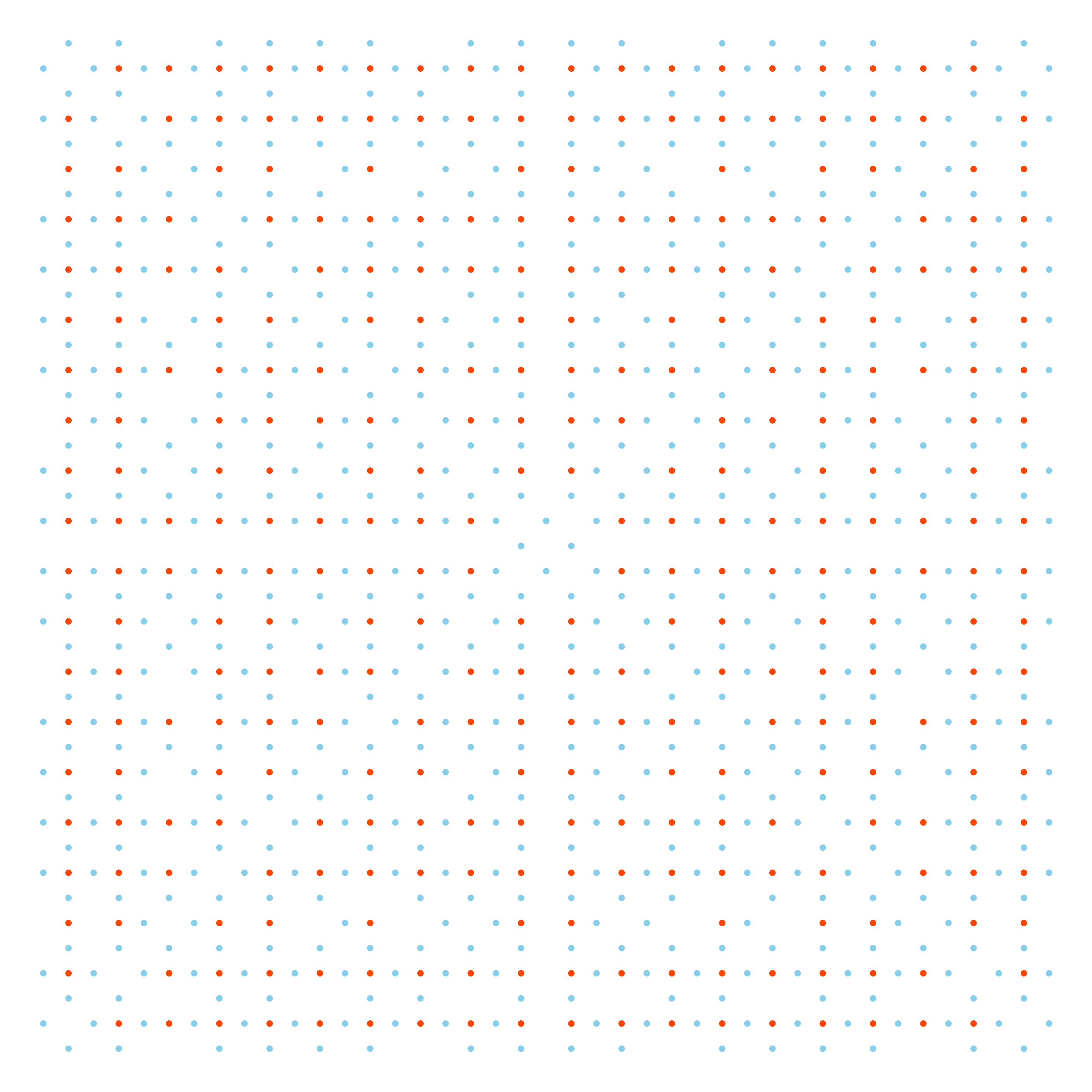}
    \end{subfigure}
    \hfill 
    \begin{subfigure}{0.49\textwidth}
        \centering
        \includegraphics[width=\textwidth]{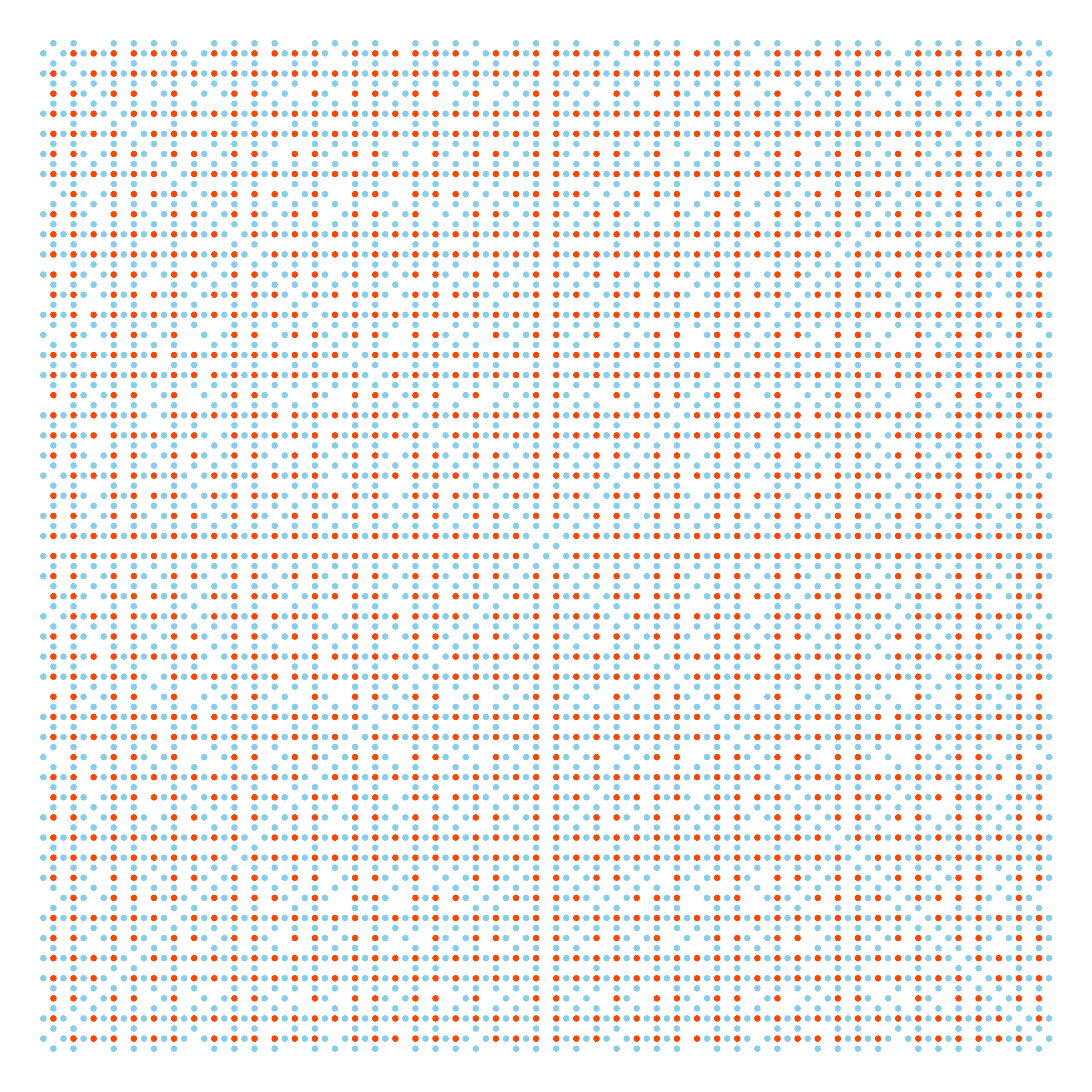}
    \end{subfigure}
    \caption{Partition $\mcV(\Z) = \mcV(\Z)_1 \sqcup \mcV(\Z)_4$ with respect
      to the Galois map $\phi(s:t) = ((s^2-t^2)^2:(s^2+t^2)^2)$, with primitivity
      defect set $\mathcal{D}(\phi) = \brk{1,4}$.}
    \label{fig:VZ-partition}
  \end{figure}

\subsection{Proof of \Cref{prop:N(Omega;h)-asymptotics} and \Cref{cor:counting}}
\label{sec:proofs} We start with the proof of the asymptotic bounds. We will abbreviate
\[\max \tilde\phi(s,t) \colonequals \max\brk{|\phi_0(s,t)|,
  |\phi_\infty(s,t)|}.\] 
\begin{proof}[Proof of \Cref{prop:N(Omega;h)-asymptotics}]

We may apply the principle of Lipschitz \cite{Davenport51} to obtain
  \begin{align}
    \label{eq:Lambda-e}
    \notag\widetilde M(h) \colonequals \#\brk{(s,t) \in \Z^2 : \max\tilde\phi(s,t)
    \leq h} \\
    = \vol(\mcR_1)\cdot h^{2/d} + O\left(h^{1/d}\right),
  \end{align}
  where $\vol\paren{\mcR_1}$ is the Lebesgue measure of the compact region $\mcR_1$
  in $\R^2$ given by $\max\brk{|\phi_0(s,t)|,|\phi_\infty(s,t)|} \leq 1$.

  In light of the partition \Cref{eq:ZZ-partition}, we see that for each
  $e\in \mathcal{D}(\phi)$ the set $\Z\cdot \mcV(\Z)_e$ has a density
  $\delta_e\in [0,1]$, and $\sum_{e\in \mathcal{D}(\phi)}\delta_e = 1$.
  Moreover, if we define
  \begin{equation*}
    \widetilde M_e(h) \colonequals \#\brk{(s,t) \in \Z\cdot \mcV(\Z)_e :
      \max\tilde\phi(s,t) \leq h},
  \end{equation*}
  then $\widetilde M_e(h) = \delta_e\cdot \widetilde M(h) + O(1)$.
  
  We apply a standard M\"obius sieve to \Cref{eq:Lambda-e} to obtain, for every
  $\varepsilon > 0$, the asymptotic
  \begin{align}
    \label{eq:Lambda-e-primitive}
    \notag \widetilde N(h) \colonequals \#\brk{(s,t) \in \mcV(\Z) : \max\tilde\phi(s,t)
    \leq h} \\
    = \frac{6}{\pi^2}\cdot\vol(\mcR_1)\cdot h^{2/d} + O_{e,\varepsilon}\left(h^{1/d+\varepsilon}\right).
  \end{align}
  Moreover, if we define
  \begin{equation*}
    \widetilde N_e(h) \colonequals \#\brk{(s,t) \in \mcV(\Z)_e :
      \max\tilde\phi(s,t) \leq h},
  \end{equation*}
  then $\widetilde N_e(h) = \delta_e\cdot \widetilde N(h) + O(1)$.  
  Consider the counting function
  \begin{equation*}
    N(h) \colonequals \#\brk{(s:t) \in \Pone(\Q) :
      \Ht(\phi(s:t)) \leq h},
  \end{equation*}
  which counts all $\Q$-rational points on $\Pone$ with respect to the height
  $\Ht$ pulled back by $\phi$. In general, we have the inequalities
  \begin{equation}
    \label{eq:4}
  \tfrac1d \cdot N(h) \leq   N(\Omega(\phi);h) \leq N(h),
  \end{equation}
  which arise from the fact that a point $Q = \phi(P) \in \Omega(\phi)$ has at
  least one rational point in the fiber $\phi^{-1}(Q)$, and at most
  $d = \deg \phi$.

  To conclude, we relate $N(h)$ to the counting functions
  $\widetilde N_e(h)$. By the definition of $\Ht$, we see that
  \begin{align}
    N(h) &= \frac12 \sum_{e\in\mathcal{D(\phi)}}\widetilde N_e(eh) \notag\\
    &= \frac12 \sum_{e \in \mathcal{D}(\phi)}\paren{
      \frac{6}{\pi^2}\cdot\vol(\mcR_1)\cdot\delta_e\cdot (eh)^{2/d} +
      O\left((eh)^{1/d+\varepsilon}\right)} \notag \\
    &= \frac{3}{\pi^2}\vol(\mcR_1)\paren{\sum_{e\in\mathcal{D}(\phi)} \delta_e\cdot e^{2/d}} \cdot h^{2/d} +
      O\left(h^{1/d+\varepsilon}\right).
  \end{align}
  In particular, the leading constant is
  \begin{equation}
    \label{constant-term}
    \delta(\phi)  = \frac{3}{\pi^2}\vol(\mcR_1)\paren{\sum_{e\in\mathcal{D}(\phi)} \delta_e\cdot e^{2/d}}.
  \end{equation}
\end{proof}

We will use \Cref{prop:N(Omega;h)-asymptotics} in the special case of a
geometrically Galois $\Q$-Belyi map $\phi$.
\begin{proof}[Proof of \Cref{cor:counting}]
  Suppose that $\phi$ is geometrically Galois, with Galois group
  $\Gal(\phi) = \Aut(\phi_\Qbar)$. Then, $\Gal(\phi)$ acts transitively and
  without stabilizers on the fibers of unramified points $Q\in \Pone(\Q)$.
  Since there are finitely many points that ramify, they do not influence the
  asymptotic count, so we ignore them. We claim that for every $Q \in
  \phi(\Pone(\Q)) = \Omega(\phi)$, we have that 
  \begin{equation*}
    \# \phi^{-1}(Q)(\QQ) = \#\Aut(\phi).
  \end{equation*}
  Indeed $\Aut(\phi) = \Aut(\phi_{\Qbar})^{\Gal_\Qbar}$, and for every
  $P \in \phi^{-1}(Q)(\QQ)$ and $\gamma\in \Aut(\phi)$, we have that
  $\gamma(P) \in\phi^{-1}(Q)(\QQ)$ as well. On the other hand, given
  $P,P' \in \phi^{-1}(Q)(\QQ)$, there exists $\gamma \in \Aut(\phi_\Qbar)$ such
  that $\gamma(P') = P$. For any $\sigma \in \Gal_\Q$, we see that
  $\gamma^\sigma(P') = \gamma(\sigma^{-1}P') = \gamma(P')$. Therefore,
  $\gamma^{-1}\gamma^\sigma$ stabilizes $P'$, which implies that
  $\gamma^{-1}\gamma^\sigma=1$, and therefore $\gamma \in \Aut(\phi)$. It
  follows that $N_\phi(h) = \#\Aut(\phi)\cdot N(\Omega(\phi);h)$, and the proof
  is complete.   In particular, the leading constant is
  \begin{equation}
    \label{constant-term-kappa}
    \kappa(\phi)  = \frac{3}{\pi^2}\dfrac{\vol(\mcR_1)}{\#\Aut(\phi)}\paren{\sum_{e\in\mathcal{D}(\phi)} \delta_e\cdot e^{2/d}}.
  \end{equation}
\end{proof}

\begin{example}[Pythagorean constant]
  \label{example:pythagorean-constant} In \Cref{sec:x2+y2-z2}, we concluded
  that for $F\colon \x^2 + \y^2 - \z^2 = 0$, we have the identity $\Omega(F) =
  \Omega(\phi)$, where $\phi\colon Z \colonequals \Proj \Q[\x,\y,\z]/(\x^2 + \y^2 - \z^2) \to
  \Pone_\Q$ is the Galois \belyi map $(x:y:z) \mapsto (x^2:z^2)$. Take the
  isomorphism $\Pone \cong Z$ given by $(s:t) \mapsto (s^2-t^2:2st:s^2+t^2)$, and rename
  $\phi$ to be the composition $\Pone \cong Z \to \Pone$, $(s:t) \mapsto
  ((s^2-t^2)^2: (s^2+t^2)^2)$.
  \begin{itemize}[leftmargin=*]
  \item Since $\max\brk{|s^2-t^2|^2, |s^2+t^2|^2} = (s^2+t^2)^2$, the region
    $\mcR_1$ is the unit disc, and $\vol(\mcR_1) = \pi$.
  \item The primitivity defect set $\mathcal{D}(\phi) = \brk{1,4}$. The
    densities are $\delta_1 = 2/3$ and $\delta_4 = 1/3$.
  \end{itemize}
  Putting this data into \Cref{constant-term}, we see that
  \begin{equation*}
    \delta(\phi) = \frac{3}{\pi^2}\cdot \pi\paren{\dfrac{2}{3} +
      \dfrac{4^{2/4}}{3}} = \dfrac{4}{\pi}.
  \end{equation*}
  Finally, since $\Aut(\phi) \cong G \cong C_2\times C_2$, we obtain $\kappa(\phi) =
  \delta(\phi)/4 = \tfrac{1}{\pi}$.
\end{example}

\section{Proof of main results}
\label{sec:proofs}

\begin{situation}
  \label{situation:proofs}
  We adopt the following notation for the rest of this section.
  \begin{itemize}[leftmargin=*]
  \item Let $\abc$ be a spherical signature (see
    \Cref{table:spherical-groups}), we do not assume that $a \leq b \leq c$.
  \item Let $\mcS$ denote a finite set of primes, and $R = \Z[\mcS^{-1}]$.
  \item Recall that $\HH^1_\mcS(\Q, G)$ denotes the Galois cohomology pointed
    set which classifies $G$-torsors over $\Spec \Q$ unramified outside of $\mcS$.
  \item For any $\Omega \subset \Pone(\Q)$, and any $h > 0$, we have the
    counting function $N(\Omega; h)$ defined in
    \Cref{situation:rat-pts-bounded-height-image-rational-map}.
  \end{itemize}
\end{situation}

Our proof follows the guidelines of the method of Fermat descent, as presented
in \cite{SAP-fermat-descent}. It consists on three steps: covering, twisting, and sieving.

\subsection{Covering} The covering is a geometrically Galois
$\Q$-\belyi map $\phi\colon \Pone_\Q \to \Pone_\Q$ with signature $\abc$. For instance we
can always start with one of the maps described by the rational functions in
\Cref{table:belyi-maps} and, since we are not assuming that $a \leq b \leq c$,
compose with an appropriate permutation $\gamma \in \PGL_2(\Q)$ of $\brk{0,1,\infty}$.
\subsection{Twisting} By \cite[Lemma 3.23]{SAP-fermat-descent}, there exists a finite
set of primes $\mcS$ for which the map $\phi$ admits an $R$-model
$\Phi\colon \Pone_R \to \Pone_R$ such that
$\Pone\abc_R \cong [\Pone_R/\Autsch(\Phi)]$. Descent theory gives the partition
  \begin{align*}
    \Pone\abc\ideal{R} &= \bigsqcup_{\tau\in \HH^1(R, \Autsch(\Phi))}
                         \Phi_\tau(\Pone_\tau(R)) \\
                       & = \bigsqcup_{\tau \in \HH^1_\mcS(\Q,\Gal(\phi))} \phi_\tau(\Pone_\tau(\Q)).
  \end{align*}
  Here, $\HH^1(R,\Autsch(\Phi))$ denotes the fppf \v{C}ech cohomology pointed
  set. It is in bijection with isomorphism classes of fppf
  $\Autsch(\Phi)$-torsor schemes $T \to \Spec R$. Restriction to the generic
  fiber induces an isomorphism
  \begin{equation*}
    \HH^1(R,\Autsch(\Phi)) \cong \HH^1_\mcS(\Q,\Gal(\phi))
\end{equation*}
of pointed sets. Note that
$\Gal(\phi) \cong \Autsch(\phi)(\Qbar) = \Aut(\phi_\Qbar)$, so the action of
the absolute Galois group $\Gal_\Q$ is the natural one. In general,
$\HH^1_\mcS(\Q,\Gal(\phi))$ is only a pointed set and not a group, since
$\Gal(\phi) \cong \tribar\abc$ as abstract groups, and the only abelian
spherical triangle group is $\tribar(2,2,2) \cong C_2 \times C_2$. Crucially,
the set $\HH^1_\mcS(\Q, \Gal(\phi))$ is finite, and classifies twists of the
\belyi map $\phi$. It is worth noting that in some cases, the source curve of a
twist $\phi_\tau\colon \Pone_\tau \to \Pone$ might be a pointless conic.
Nevertheless, since the equations
\begin{align*}
  \x^2 + \y^2 - \z^c &= 0, \quad (c \geq 2) \\
  \x^2 + \y^3 - \z^3 &= 0,\\
  \x^2 + \y^3 - \z^4 &= 0,\\
    \x^2 + \y^3 - \z^5 &= 0,
\end{align*}
all have primitive integral solutions, we know that $\Omega\abc \neq \emptyset$,
and there will always be at least one twist for which $\Pone_\tau(\Q)
\neq \emptyset$.

\subsection{Sieving} Combining the partition above with
\Cref{cor:counting}, we obtain
  \begin{equation*}
    N(\Omega_\mcS\abc; h) = \sum_\tau N(\Omega(\phi_\tau);h),
  \end{equation*}
  where the sum ranges over all the $\tau\in \HH^1_\mcS(\Q, \Gal(\phi))$ for which
  $\Pone_\tau$ is isomorphic to $\Pone_\Q$. To sieve out the excess of elements in
  $\Pone\abc\ideal{R}$ not corresponding to points in
  $\Omega\abc = \Pone\abc\ideal{\Z}$, we show that we can restrict to certain
  subsets $T(F) \subset T\abc \subset \HH^1_\mcS(\Q,\Gal(\phi))$ to cover all of
  $\Omega(F)$ and $\Omega\abc$. The proofs of both \Cref{thm:main-belyi} and
  \Cref{thm:main-fermat} (in the special case of simplified equations
  (\Cref{def:fermat-coefficients})) will follow immediately from the following
  lemma.

  \begin{lemma}
    \label{lemma:simplified-partitions}
    Fix a possibly empty subset $\mcT \subset \mcS$. Take a $\mcT$-simplified
    Fermat equation $F\colon \gfe = 0$. Then, there is a finite subset
    $T(F) \subseteq \HH^1_\mcS(\Q,\Gal(\phi))$ such that
    \begin{equation}
      \label{eq:Omega-F-partition}
      \Omega(F) = \bigsqcup_{\tau \in T(F)}\phi_\tau(\Pone_\tau(\Q)).
    \end{equation}
    Moreover, defining $T\abc$ as the disjoint union of the sets $T(F)$, as $F$
    ranges over all $\emptyset$-simplified Fermat equations of signature $\abc$,
    we have
    \begin{equation}
          \label{eq:Omega-abc-partition}
          \Omega\abc = \bigsqcup_{\tau \in T\abc}\phi_\tau(\Pone_\tau(\Q)).
     \end{equation}
   \end{lemma}
   \begin{proof}
     Any geometrically Galois $\Q$-\belyi map
     $\phi\colon \Pone_\Q \to \Pone_\Q$ of signature $\abc$ is given
     by a rational function
     \begin{equation*}
       \dfrac{\phi_0}{\phi_\infty} = 1 + \dfrac{\phi_1}{\phi_\infty} \in \kk(\Pone_\Q),
     \end{equation*}
     where
     \begin{enumerate}[leftmargin=*, label=(\roman*)]
     \item $\phi_0, \phi_1,\phi_\infty\in \Z[\s,\t]$ are homogeneous of degree $\#\tribar\abc$,
     \item $\gcd(\phi_0,\phi_\infty) = \gcd(\phi_1, \phi_\infty) = 1$, and
     \item we can write
       \begin{align*}
         \phi_0(\s,\t) &= C_0\cdot X(\s,\t)^a, \\
         \phi_1(\s,\t) &= C_1\cdot Y(\s,\t)^b,\\
         \phi_\infty(\s,\t) &= C_\infty\cdot Z(\s,\t)^c,
       \end{align*}
       for unique polynomials $X,Y,Z \in \Z[\s,\t]$, and a unique triple
       $(C_0,C_1,C_\infty)$ of $\mcS(\phi)$-simplified Fermat coefficients,
       where $\mcS(\phi)$ is an explicit set of bad primes.
     \end{enumerate}
     We denote this triple by $\coeff(\phi)$. Observe that for any
     $Q\in \Pone(\Q)$, we have that $\coeff(\phi) = \coeff(\phi(Q))$.

     Returning to the situation of this section, to each cohomology class
     $\tau$ we can associate the $\mcS$-simplified Fermat coefficient triple
     $ \coeff(\phi_\tau)$. If $F$ is $\mcT$-simplified, then it is also
     $\mcS$-simplified. Moreover, for every primitive integral solution
     $(x,y,z)$ to $F$, the point $j(x,y,z) \in \Pone(\Q)$ is in
     $\Omega_\mcS\abc$. Define
     \begin{equation*}
       T(F) \colonequals \brk{\tau \in \HH^1_\mcS(\Q,\Gal(\phi)) :
         \coeff(\phi_\tau) = (A,B,C)}.
     \end{equation*}
   \end{proof}

   To finish the proof of \Cref{thm:main-fermat}, we must consider the case of
   non-simple equations. To guide our intuition, consider the equation
   $F'\colon 25\x^2 + \y^2 = \z^2$. Our strategy is to use the
   \emph{simplification} $F \colon \x^2 + \y^2 = \z^2$ to deduce the asymptotic
   result for $F'$ from that of $F$. In this case, the $\Q$-isomorphism of nice
   curves $C \to C', \quad (x:y:z) \mapsto (x/5:y:z)$ enables this translation.
   The idea is that the congruence condition $x \equiv 0 \md 5$ cuts out a positive
   proportion of the primitive integral solutions to the Pythagorean equation,
   and only the constant term in the asymptotic will change.

   Start with a non-simple equation $F'\colon A'\x^a + B'\y^b + C'\z^c = 0$.
   Without loss of generality, we may assume that $\gcd(A',B',C') = 1$. In this
   case, we can write
   \begin{equation*}
     A' = A\cdot A_0^a, \quad B' = B\cdot B_1^b, \quad C' = C\cdot C_\infty^c,
   \end{equation*}
   to obtain a $\mcT$-simplified coefficient triple $(A,B,C)$, where $\mcT$ is
   the set of primes dividing $A'\cdot B'\cdot C'$. The Fermat equation
   $F\colon \gfe = 0$ is the \cdef{simplification} of $F'$. From
   \Cref{lemma:simplified-partitions}, we have a partition
   \begin{equation}
     \label{eq:Omega-F-partition}
     \Omega(F) = \bigsqcup_{\tau \in T(F)} \Omega(\phi_\tau),
   \end{equation}
   where each $\phi_\tau$ is a geometrically Galois $\Q$-\belyi map
   $\Pone_\Q \to \Pone_\Q$ of signature $\abc$. Let $\phi$ be one of these maps. We
   have seen that $\phi$ corresponds to a rational function
   \begin{equation*}
     \phi = \dfrac{A\cdot X(\s,\t)^a}{C\cdot Z(\s,\t)^c} = 1 + \dfrac{B\cdot
       Y(\s,\t)^b}{C\cdot Z(\s,\t)^c}.
   \end{equation*}
   To conclude, we use a clever argument of Beukers \cite[Proof of Theorem
   1.5]{Beukers98}. Consider the polynomial map
   \begin{equation}
     \label{eq:alpha}
     \alpha\colon \Q^2 \to \Q^3, \quad (s,t) \mapsto \paren{\dfrac{X(s,t)}{A_0},
       \dfrac{Y(s,t)}{B_1}, \dfrac{Z(s,t)}{C_\infty}}.
   \end{equation}
   We use $\alpha$ to define a lattice of rank two generated by the points whose image is integral
   \begin{equation*}
     \Lambda(\alpha) \colonequals \mathrm{Span}_\Z\brk{(s,t)\in \Q^2 :
       \alpha(s,t)\in \Z^3}.
   \end{equation*}
   Choose an integral basis $\brk{\vec\alpha_1, \vec\alpha_2}$ for
   $\Lambda(\alpha)$, and define
   \begin{equation*}
     \phi' = \dfrac{A\cdot X(\s\vec\alpha_1 + \t\vec\alpha_2)^a}{C\cdot Z(\s\vec\alpha_1 + \t\vec\alpha_2)^c} = 1 + \dfrac{B\cdot
       Y(\s\vec\alpha_1 + \t\vec\alpha_2)^b}{C\cdot Z(\s\vec\alpha_1 + \t\vec\alpha_2)^c}.
   \end{equation*}
   Applying this construction to every $\phi_\tau$ appearing in
   \Cref{eq:Omega-F-partition}, we obtain the partition
   \begin{equation*}
       \Omega(F') = \bigsqcup_{\tau \in T(F)} \Omega(\phi'_\tau),
     \end{equation*}
     from which we conclude the proof.

\bibliographystyle{amsalpha}

\providecommand{\bysame}{\leavevmode\hbox to3em{\hrulefill}\thinspace}
\providecommand{\MR}{\relax\ifhmode\unskip\space\fi MR }
\providecommand{\MRhref}[2]{%
  \href{http://www.ams.org/mathscinet-getitem?mr=#1}{#2}
}
\providecommand{\href}[2]{#2}

\end{document}